\documentclass[10pt,twoside]{article}
\usepackage{amsfonts,amsmath,amssymb,amsopn,amsthm,mdwlist,wasysym}

\input diagxy

\DeclareMathOperator{\colim}{colim}

\newcommand{\C}{{\mathbb C}}
\newcommand{\N}{{\mathbb N}}

\newcommand{\fdhilb}{{\bf fdhilb}}

\newcommand{\Ban}{{\bf Ban}}
\newcommand{\Vect}{{\bf Vect}}
\newcommand{\Bil}{{\bf Bil}}

\renewcommand{\C}{\mathbb C}
\newcommand{\lbracky}{\left[}
\newcommand{\rbracky}{\right]}

\renewcommand{\leq}{\leqslant}

\newcommand{\too}{\longrightarrow}

\newcommand{\oto}[1]{\stackrel{\underrightarrow{\stackrel{#1}{\,\,\,\,\hphantom{\too}}}}{}}  

\theoremstyle{definition}

\newtheorem{cor}{Corollary}
\newtheorem{prop}{Proposition}
\newtheorem{lemma}{Lemma}
 
\author{\sc Brian Day\footnote{Department of Mathematics, Macquarie University, NSW 2109, Australia.}}
\date{April 24, 2009}
\title{A *-autonomous Category of Banach Spaces -- Corrected}

\begin{document}

\maketitle

\begin{abstract}

We describe a $\C$-linear additive *-autonomous category of Banach spaces.
Please note that a correction has been appended to the original version 1 which is maintained here for reference. Also,
a proposed example of a *-autonomous category of topological $\C$-linear spaces has been added to version 2.

\end{abstract}

\section{Introduction}

First, we describe an elementary *-autonomous monoidal category: \[ \Ban( \mathbb N ) \subset \Ban \] where $\Ban$ is the usual symmetric monoidal
closed category of complex Banach spaces and linear contractions, and $\Ban( \mathbb N )$ is the replete full subcategory of $\Ban$ determined
by the Hilbert spaces $\ell_2(X)$ for $X \in \{ 0,1,2, \ldots, \mathbb N\}$. We also consider the additive aspects of the corresponding
$\mathbb C$-linearisations: \[ \mathbb C \Ban( \mathbb N ) \subset \mathbb C \Ban \]

The resulting $\mathbb C$-linear and additive *-autonomous category $\mathbb C \Ban (\mathbb N)$ is analagous to the ``completion'' of any
(symmetric) compact closed category to a *-autonomous monoidal category by adding both an object-at-zero and an object-at-infinity. In this case,
the object-at-zero is $\ell_2(0)$, while the object-at-infinity is $\ell_2(\mathbb N)$. Another example is the multiplicative group of positive
real numbers with both $0$ and $\infty$ adjoined, which similarly forms a *-autonomous category (with $r \otimes \infty = \infty$ for $r \neq 0$
and $0 \otimes \infty = 0$). 

This article probably contains nothing that is essentially new.

\section{The *-autonomous Category $\mathbb C \Ban(\mathbb N)$}

First we note that a norm-decreasing linear map: \[ \bigoplus^n \mathbb C \too B \qquad (n \mbox{ finite}) \]
corresponds to a set $\{b_1, b_2, \ldots, b_n\}$ of $n$ points in $B$ such that: \[ ||b_1 + b_2 + \cdots + b_n||^2 \leq n \]
Then we deduce that the canonical map: \[ \bigoplus^n \mathbb C \too \left[ B, \bigoplus^n B \right] \] is norm-decreasing and so gives a map:
\[ \left( \bigoplus^n \mathbb C \right) \otimes B \too \bigoplus^n B \] for all Banach spaces $B$.

In particular, taking $B = \bigoplus^m \mathbb C$, we obtain: \[ \mathbb C^n \otimes \mathbb C^m \too \mathbb C^{n \times m} \] if we write 
$\mathbb C^n = \bigoplus^n \mathbb C$, etc.

We also obtain a natural transformation $\alpha_B$: \[ \bfig

\node a(-1000,0)[\Ban(\C^n, \lbracky \C^m, B \rbracky)]
\node b(0,0)[\Ban(\C^{n \times m}, B)]
\node c(0,-500)[\Bil(\C^n \times \C^m, B)]

\arrow[a`b;\alpha_B]
\arrow[b`c;]
\arrow/>->/[a`c;]

\efig \]
which is thus a monomorphism. So, by the Yoneda lemma, and the natural isomorphism: 
\[ \Ban\left(\C^n \otimes \C^m,B\right) \cong \Ban\left(\C^n, \left[\C^m,B\right]\right) \] we have an epimorphism: \[ \C^{n \times m} \too
\C^n \otimes \C^m \] in $\Ban$. But $\C^n \otimes \C^m$ is a topological completion of the algebraic tensor product of $\C^n$ and $\C^m$. Hence
we have:

\begin{prop} $\C^n \otimes \C^m \oto{\cong} \C^{n \times m}$ in $\Ban$\end{prop}

\begin{prop} We also have $\colim_{n \subset X} \bigoplus^n \C \cong \ell_2(X) = \bigoplus_X \C$ in $\Ban$ for all sets $X$. \end{prop} 
The proof is straightforward.

\begin{cor} We have $\ell_2(X) \otimes \ell_2(Y) \cong \ell_2(X \times Y)$ in $\Ban$ for all $X,Y \in \{0,1,2,\ldots,\mathbb N\}$\end{cor}
The proof follows from Propositions 1 and 2, and the preservation of colimits by $\otimes$ in $\Ban$. 

Each object in $\Ban(\mathbb N)$ is
reflexive as a Banach space, so we get: \begin{prop} The monoidal replete full subcategory \[ \Ban(\mathbb N) \subset \Ban \] is a *-autonomous
monoidal category.\end{prop}
Proof: \begin{align*}
	\Ban(\mathbb N)(A \otimes B, C) 
		&= \Ban(A \otimes B, C) \\
		&\cong \Ban\left(B \otimes A, \left[\left[C,\C\right],\C \right] \right) \\
		&\cong \Ban\left(B, \left[A, \left[\left[C,\C\right],\C \right] \right] \right) \\
		&\cong \Ban\left(B, \left[\left[C,\C\right],\left[A,\C \right] \right] \right) \\
		&\cong \Ban\left(B \otimes \left[C,\C\right], \left[A,\C \right] \right) \\
		&= \Ban(\mathbb N)\left(B \otimes C^*, A^* \right) \\
\end{align*} where $A^*$ denotes $\left[A,\C\right]$, etc., so there is a natural isomorphism: \[ \Ban(\mathbb N)(A \otimes B, C) \cong 
\Ban(\mathbb N)(A,(B \otimes C^*)^*) \]

Now consider the $\C$-linearisations \[ \C \Ban( \mathbb N) \subset \C \Ban \] where $\C \Ban(\mathbb N)$ is automatically 
*-autonomous from Proposition 3. First, note that the (finite) direct sum $A \oplus B$ of two Banach spaces $A$ and $B$ (which is the $\C$-vector space product $A \times B$
with the norm $||(a,b)|| = \sqrt{||a||^2 + ||b||^2}$ ) becomes a biproduct in $\C \Ban$. This is fairly immediate from the fact that we have
bijective contraction maps \[ A + B \too A \oplus B \qquad \qquad A \oplus B \too A \times B \]
in $\Ban$, where the coproduct $A + B$ in $\Ban$ is the $\C$-vector space product $A \times B$ with the norm \[ ||(a,b)|| = ||a|| + ||b|| \] while the 
product $A \times B$ in $\Ban$ is the $\C$-vector space product $A \times B$ with the norm \[ ||(a,b)|| = \max\left(||a||,||b||\right) \]
Then the canonical diagram: \[ \bfig
\square<2000,500>[\C\Ban(A,B\oplus C)`\C\Ban(A,B)\oplus\C\Ban(A,C)`\C\Ban(A,B \times C)`\C\Ban(A,B)\otimes\C\Ban(A,C);i`\rm{can}`\rm{can}`\cong]
\efig \]
commutes in $\Vect_\C$, so $i$ is in injection, and is also a retraction, hence, it is an isomorphism. Similarly, the diagram \[ \bfig
\square<2000,500>[\C\Ban(A \oplus B,C)`\C\Ban(A,C)\oplus\C\Ban(B,C)`\C\Ban(A+B,C)`\C\Ban(A,C)\otimes\C\Ban(B,C);j`\rm{can}`\rm{can}`\cong]
\efig \] 
commutes in $\Vect_\C$, so then $j$ also is an isomorphism. Thus:

\begin{prop} The $\C$-linear category $\C\Ban(\mathbb N)$ is additive and *-autonomous. \end{prop}

Finally, we note that the braid groupoid $\mathbb B$ generates a convolution functor cateogry \[ [\mathbb B,\C \Ban(\mathbb N)]_{\rm f.s.} \]
consisting of the functors \[ F : \mathbb B \too \C\Ban(\mathbb N) \] of finite support, and the natural transformations between them. Here the
(lax) tensor product is given by: \[ F * G(l) = \bigoplus_{m,n} \mathbb B(m+n,l) \cdot (F(m) \otimes G(n)) \] where $\mathbb B(m,n) \cdot B$ is
defined to be the countable direct sum \[ \bigoplus_{\mathbb B(m,n)} B \] in $\Ban$ for each Banach space $B$. This yields a lax monoidal
functor category (with lax unit $\mathbb B(0,-) \cdot \C$). The functor category \[ [\mathbb B, \C \Ban(\mathbb N)]_{\rm f.s.} \] also has on it
the pointwise tensor product \[ F \otimes G(l) = F(l) \otimes G(l) \] which is $\C$-linear, additive, and *-autonomous.

\section*{Correction}
The canonical contraction map \[ l_2(X) \otimes l_2(Y) \too l_2(X \times Y) \]
in Proposition 1 and Corollary 1 is not in general an isomorphism in $\Ban$ (as pointed out by Dr. Yemon Choi) which negates the 
overall purpose of the article.

One ``alternative'', that seems not too significant or useful, is to consider the compact-closed category $\fdhilb$ of
finite-dimensional Hilbert spaces, and adjoin an abstract terminal object called ``$l_2(\N)$''. One then obtains a symmetric monoidal
*-autonomous category extending the structure of $\fdhilb$ on defining also: \begin{align*}
	l_2(X) \otimes l_2(\N) &= l_2(\N) \mbox{ for } X \neq 0 \\
	l_2(0) \otimes l_2(\N) &= l_2(0)
\end{align*}

with $l_2(0)^* = l_2(\N)$ and $l_2(\N)^* = l_2(0)$ instead of $l_2(0)^* = l_2(0)$.

\section{A *-autonomous category of $\C$-linear spaces}

We shall somewhat overstate the main result we need:

\begin{prop} If $\Pi B_x$ is a topological product of complex Banach spaces and $V$ is a $\C$-subspace (subspace topology) of $\Pi B_x$, then
any continuous $\C$-linear map from $V$ to $\C$ extends to $\Pi B_x$.\end{prop}

The proof follows from that of Kaplan\cite{K} Theorem 1, together with the Hahn-Banach Theorem. We need the special case $B_x = \C$ for all
$x \in X$.

Let $\Vect(\C)$ denote the monoidal closed category of topological $\C$-linear spaces and continuous $\C$-linear maps, with the pointwise
internal-hom $(-,-)$ and tensor product. Let $\mathcal P(\C)$ denote the full reflective subcategory of $\Vect(\C)$ determined by the
$\C$-subspaces of powers of $\C$. Then $\mathcal P(\C)$ is complete and closed under exponentiation in $\Vect(\C)$, hence is monoidal
closed by \cite{D}, and in fact has all small colimits since $\Vect(\C)$ does.

\begin{lemma} The canonical map $V \too ((V,\C),\C)$ is a surjection for all $V \in \Vect(\C)$. \end{lemma}
\begin{proof} Consider the following diagram, which is natural in $V$:
\[ \bfig \square<1000,500>[\hom(\C,V) \cdot (\C,\C)`V`(\C^{\hom(\C,V)},\C)`((V,\C),\C);h_V`f_V`{\rm can}`g_V] \efig \]
which is easily seen to commute by applying the Yoneda lemma to $V \in \Vect(\C)$, where ``$\cdot$'' denotes copower in $\Vect(\C)$. Then
$g_V$ is a surjection by Proposition 5, and $f_V$ is a surjection because any continuous $\C$-linear map \[ \C^{\hom(\C,V)} \too \C \]
factors (uniquely) through projection onto some finite power $\C^n$ where $n \subset \hom(\C,V)$. So $V \too ((V,\C),\C)$ is a surjection.\end{proof}

Now the objects of $\mathcal P(\C)$ are precisely the subspaces $V \leq (W,\C)$ for some $W \in \Vect(\C)$, so, since
\[ \bfig \ptriangle/->`->`-->/<1000,500>[V`((V,\C),\C)`(W,\C);\hbox{can}`\leq`] \efig \]
commutes for such a $V$, we get $V \cong ((V,\C),\C)$ by lemma 1. Hence, from the proof of Proposition 3, we get:
\begin{prop} $\mathcal P(\C)$ is a $\C$-linear *-autonomous category.\end{prop}

\end{document}